\documentclass[12pt]{article}
\usepackage{amsmath}
\usepackage{amssymb}
\usepackage{amsthm}
\usepackage{graphicx}
\usepackage{authblk}
\newcommand{\ov}{\overline}

\newcommand{\cS}{{\cal S}}

\newcommand{\cH}{{\cal H}}

\newtheorem{thm}{Theorem}
\newtheorem*{thm*}{Theorem}
\newtheorem{lem}{Lemma}


\title{Nordhaus-Gaddum-type problems for lines in hypergraphs}
\author[1]{Xiaomin Chen}
\author[2]{Peihan Miao}
\affil[1]{Shanghai Jianshi LTD}
\affil[2]{University of California -- Berkeley}
\begin{document}

\maketitle

\begin{abstract}
We study the number of lines in hypergraphs in a more symmetric setting,
where both the hypergraph and its complement are considered.
In the general case and in some special cases, the lower bounds
on the number of lines are much higher than their counterparts in
single hypergraph setting or admit more elegant proofs.
We show that the minimum value of product of the number of lines in
both hypergraphs on $n$ points is easily determined as $\binom{n}{2}$;
and the minimum value of their sum is between $\Omega(n)$ and $O(n \log n)$.
We also study some restricted classes of hypergraphs; and determine the
tight bounds on the minimum sum when the hypergraph is derived from an Euclidean space, a real projective plane, or a tree.
\end{abstract}
A {\em hypergraph\/} is an ordered pair $(V,\cH)$ such that $V$ is a set
and $\cH \subseteq 2^V$ is a family of subsets of $V$; elements of $V$ are the {\em
  vertices\/} of the hypergraph and members of $\cH$ are its {\em
  hyperedges;\/} a hypergraph is called {\em $k$-uniform\/} if each of
its hyperedges consists of $k$ vertices; i.e. $\cH \subseteq \binom{V}{k}$.
In studying problems related to a De Bruijn-Erd\H{o}s theorem,
Chen and Chv\'{a}tal~\cite{CC} defined lines in 3-uniform hypergraphs.
Specifically, for any $u$, $v \in V$, the line $\ov{uv}$ is defined as
\[
\overline{uv}\;=\;\{u,v\} \cup \left\{p:\{u,v,p\}\in \cH\right\}.
\]
It is shown that, when $V$ is not a line, there are at least $(2 - o(1)) \log_2 n$ lines \cite{ABCCCM},
and there are examples showing that the number of lines can be as few as $c^{\sqrt{\log n}}$ for some constant $c$~\cite{CC}.

In this work, we consider the lines in both the 3-uniform hypergraph $(V, \cH)$ and its complement
\[(V, \ov{\cH}) = \left(V, \binom{V}{3} \setminus \cH\right).\]
We focus on the lower bound on the sum and the product of the number of lines in both hypergraphs.
In graph theory this falls into the category of Nordhaus-Gaddum type problems (\cite{NG}, \cite{AH}).
We show that although the number of lines in one hypergraph can be quite small,
the arithmetic and geometric mean of the number of lines from both hypergraphs is much bigger. And in some situations, the more symmetric version admits more elegant solutions.

The tight bound on the product of the number of lines from both hypergraphs is quite easy to determine. Our main results are about the sum of the number of lines from both hypergraphs.
We prove in Theorem \ref{thm.easy_bound} that the sum is at least in the order of $n$; and in Theorem \ref{thm.upper_bound} we
construct examples where the sum can be as small as $O(n \log n)$.
In Section \ref{sect.small_inter} we study the systems where one hypergraph has lines
with small intersections; and prove that the sum is always in the order of $\Theta(n^2)$
(Theorem \ref{thm.small_inter}). In the special case where two lines intersect at no more than 1 point, we show that the bound on the sum is tight around $n^2 / 6 \pm O(n)$.
In Section \ref{sect.tree}, we completely solve the case when one hypergraph is derived from a metric space that is in turn derived from a tree. Theorem \ref{thm.tree} gives the lower bound on the sum and characterizes all the tight examples.

We point out a small difference in this work from the previous ones.
In the symmetric versions,
we no longer have the requirement that $V$ itself to be excluded from being
a single line,
as one usually does in the De Bruijn-Erd\H{o}s settings.

\section{Definitions and examples}\label{sect.intro}

In most situations, we would like to view the graph and its complement symmetrically.
For this purpose, we consider a bi-colouring of $\binom{V}{3}$,
where edges in $\cH$ are coloured red and those in $\ov{\cH}$ are blue.

A {\em bi-coloured system} is a pair $\cS = (V, f)$ where $V$ is a set of points,
and $f$ is a function defining the colour of every triple, i.e., $f: \binom{V}{3} \rightarrow \{R, B\}$.
For any $\{u, v\} \in \binom{V}{2}$, we define the red line $R_\cS(uv)$ as
\[R_\cS(uv) = \{u, v\} \cup \left\{ p: f(\{u, v, p\}) = R\right\}\]
and similarly, the blue line is defined as
\[B_\cS(uv) = \{u, v\} \cup \left\{ p: f(\{u, v, p\}) = B\right\}.\]
Thus any line is a subset of points. If a red (resp. blue) line $L = R_\cS(uv)$
(resp. $L = B_\cS(uv)$), we say $L$ is {\em generated} by the pair $u$ and $v$.
We may have different pairs generate the same line, so the number of lines
in any colour might be less than $\binom{|V|}{2}$.

We define
\[ m_\cS(R) = \mbox{the number of distinct red lines}, \]
\[ m_\cS(B) = \mbox{the number of distinct blue lines}, \]
and
\[m_\cS = m_\cS(R) + m_\cS(B),\;  m^*_\cS =  m_\cS(R) m_\cS(B).\]
Note that we choose to count red lines and blue lines separately, even if a red line
and a blue line happen to be the same set.

Further, we define the 3-uniform hypergraph $\cH_\cS$ as
\[ \cH_\cS = \left\{ \{a, b, c\} \in \binom{V}{3} : f(\{a, b, c\}) = R \right\}. \]
When there is no ambiguity, we often omit the subscript $\cS$.

We give some simple examples here.
When $\cH$ is the complete 3-uniform hypergraph, $m(R) = 1$ and $m(B) = \binom{n}{2}$ where $n$ is the size of $V$.
When $\cH$ forms a Steiner triple system, $V \setminus \{v\}$ is a blue line for each $v \in V$,
thus $m(B) = n$, on the other hand $m(R) = \binom{n}{2} / 3$.
When $\cH$ is the projective plane of order $k > 1$ with $n = k^2 + k + 1$ points,
$m(R) = n$; and for each pair $u, v \in V$, the blue line
$B(uv) = V \setminus \{ \mbox {points on the projective line } uv \} \cup \{u, v\}$,
the number of blue lines $m(B) = \binom{n}{2}$.

We say that $\cS$ is {\em derived} from a finite metric space $(V, \rho)$ if $\cH$ consists of
all the 3-element sets $\{a, b, c\}$ such that $\rho(a, b) + \rho(b, c) = \rho(a, c)$.

Each metric space is derived from a positively edge-weighted connected graph
$(G, w)$, where $\rho(a, b)$ can be viewed as the length of any shortest path between
$a$ and $b$ in graph $G$.
It is easy to see that for each metric space, there is a minimum graph deriving it.

One special case of metric spaces is the spaces derived from a connected (unweighted) graph $G$.
We may think all the edges having length 1, thus the distance between two vertices is the least
number of steps one needs to take going from one to the other.
We further say that $\cS$ is derived by a graph $G$ if it is derived by a metric space that is in turn derived from $G$.

\section{General hypergraphs}\label{sect.hypergraph}

For each $n$, we define
\[m(n) := \min \{m_\cS : \cS = (V, f) \;\mbox{is a bi-coloured system with $|V| = n$} \}, \]
\[m^*(n) := \min \{m^*_\cS : \cS = (V, f) \;\mbox{is a bi-coloured system with $|V| = n$} \}.\]

In this section, we show that $m^*(n)$ is easily determined; and we give bounds for $m(n)$.

\begin{lem}\label{lem.simple}
If $\{a_i,b_i\}$ ($i = 1, ..., t$) are $t$ pairs that generates the same red line
(i.e., $R(a_ib_i)$ are all the same set),
then $B(a_ib_i)$ are $t$ different blue lines,
and $B(a_ib_i) \cap X = \{a_i, b_i\}$, where $X = \cup_{i=1}^t \{a_i, b_i\}$.
\end{lem}

\begin{proof}
Pick any $i$, $X \subseteq R(a_ib_i)$, so the triple $\{a_i, b_i, x\}$
is red for any $x \in X \setminus \{a_i, b_i\}$,
thus $B(a_ib_i) \cap X = \{a_i, b_i\}$.
\end{proof}

\begin{thm}\label{thm.easy_bound}
For any $n$,

(a) $m(n) \geq 2 \sqrt{\binom{n}{2}}$

(b) $m^*(n) = \binom{n}{2}$;
and $m^*_\cS = \binom{|V|}{2}$ if and only if
$\cH_\cS$ is the complete 3-uniform hypergraph or the empty hypergraph.
\end{thm}

\begin{proof}
Let $t$ be the maximum number of pairs of vertices that generate the same red line.
Then $m(R) \geq \binom{n}{2} / t$
and $m(B) \geq t$ by Lemma \ref{lem.simple}.
Thus $m \geq 2 \sqrt{\binom{n}{2}}$ and $m^* \geq \binom{n}{2}$.

If the equality holds in (b), each red line is generated by exactly $t$ pairs of vertices.
If $t=1$, then $m(R) = \binom{n}{2}$, and $m(B) = 1$ only when $\cH$ is the empty hypergraph.
So we may assume $t > 1$.

The equality clearly holds if $\cH$ is empty or complete.
Otherwise, pick a red line $L$ (it is generated by $t$ pairs),
and define the graph of generating pairs of $L$,
\[ G = (V, E) = \left( V, \{\{a, b\} : R(ab)=L \} \right). \]
Let $X$ be the set of vertices in $G$ with positive degree
-- the set of vertices each of whom generates $L$ with another vertex.
By Lemma \ref{lem.simple}, we have $t$ distinct blue lines
$B(uv)$ for each $\{u,v\} \in E$, where $B(uv) \cap X = \{u, v\}$

{\em Case 1.} $G$ is not a clique on $X$, with at least one missing edge $xy$,
then by Lemma \ref{lem.simple}, $\{x, y\} \subseteq B(xy) \cap X$ implies that $B(xy)$ is different from all those $t$ blue lines.

{\em Case 2.} $G$ is a clique on $X$ and $X \neq V$.
Pick any point $y \in V \setminus X$ and any point $x \in X$.
If $y \in L$, then $\{x, x', y\}$ is red for each $x' \in X \setminus \{x\}$,
and $B(xy) \cap X = \{x \}$; otherwise $\{x, x', y\}$ is blue for each $x' \in X \setminus \{x\}$,
and $B(xy) \cap X = X$, which has size at least $3$ (because $t > 1$).
In any of the above cases, $|B(xy) \cap X| \neq 2$, by Lemma \ref{lem.simple} it is
different from the $t$ blue lines.

In either case, we have at least $\binom{n}{2} / t$ red lines and $t+1$ blue lines.
The equality in (b) does not hold.
\end{proof}

Note that the lower bound on $m^*_\cS$ is easy and
the trivial tight example is included in most situations we consider in this article.
Also note that the bound on $m_\cS$ is already much higher than the situation when we
consider one hypergraph alone, where the number of lines can be as few as $c^{\sqrt{\log n}}$.

Next we give examples showing that $m(n)$ is also bounded from above.

\begin{thm}\label{thm.upper_bound} $m(n) \leq n \lceil \log_2 n \rceil.$
\end{thm}

\begin{proof}
We claim that whenever $n_1 + n_2 = n$, we have
\[ m(n) \leq m(n_1) + m(n_2) + n.\]
In particular,
\[ m(n) \leq m\left( \left\lfloor \frac{n}{2} \right\rfloor \right) + m\left( \left\lceil \frac{n}{2} \right\rceil \right) + n \]
together with the base values give us the desired bound.

To justify the claim, consider two systems $\cS_i = (V_i, f_i)$ ($i = 1, 2$) where
$|V_i| = n_i$ and $m_{\cS_i} = m(n_i)$, and $V_1 \cap V_2 = \emptyset$. Now define $\cS = (V, f)$ where $V = V_1 \cup V_2$, and $f$ restricted on $\binom{V_i}{3}$ is identical to $f_i$. We still need to define $f(\{a, b, c\})$ whenever $a, b, c$ do not come from the same part. In this case, $f(a, b, c)$ is red is two of the points are from $V_1$, otherwise blue.

It is easy to check that the system contains the following lines:
\begin{itemize}
\item a red line $V_1 \cup \{ v \}$ for each $v \in V_2$, these are generated by $v$ and any point in $V_1$.
\item a blue line $\{ v \} \cup V_2$ for each $v \in V_1$, these are generated by $v$ and any point in $V_2$.
\item a red line $L \cup V_2$ for each red line $L$ in $\cS_1$, generated by two points in $V_1$.
\item a blue line $L$ for each blue line $L$ in $\cS_1$, generated by two points in $V_1$.
\item a red line $L$ for each red line $L$ in $\cS_2$, generated by two points in $V_2$.
\item a blue line $V_1 \cup L$ for each blue line $L$ in $\cS_2$, generated by two points in $V_2$.
\end{itemize}

In total, the sum of the number of red lines and blue lines is $m(n_1) + m(n_2) + n$.

\end{proof}

We may give an explicit construction as follows.
We organize a (almost) balanced rooted binary tree, the elements of $V$ being leaves.
For any three leaves $a$, $b$, and $c$, we find their common anscestor $z$. We colour
$\{a, b, c\}$ red if and only if two of them are from the left subtree of $z$.
It is easy to see that, a line is determined by any inner node $z$ and one leaf in the
subtree rooted at $z$, therefore we get the $n \log n$ bound.

Our computer program finds $m(n)$ for small $n$ as follows,

\begin{center}
\begin{tabular}{ | c | c | c | c | c | c | c |}
  \hline
  $n$    & 2 & 3 & 4 & 5  & 6  & 7 \\ \hline
  $m(n)$ & 2 & 4 & 7 & 11 & 14 & $\leq 17$ \\ \hline
\end{tabular}
\end{center}

\section{Red lines with small intersections}\label{sect.small_inter}

We prove that when the intersection of red lines is bounded, then $m_\cS$ is $\Omega(n^2)$.
One special case we solve almost completely is when any two red lines intersect at no more than 1 point. Such a system conforms well to geometric intuitions and was defined
as {\em strongly geometric dominant} (\cite{CHMY}).
This includes the case when the system is derived from an Euclidean space,
or when the red lines are the lines from a real projective plane,
or when the red lines form a Steiner triple system.

\begin{thm}\label{thm.small_inter} If $\cS$ is a system in which the intersection of any two red lines intersect in at most $k$ points for some positive integer $k$,
then
\[m_\cS \geq \frac{\binom{n}{2}}{\binom{k+2}{2}}.\]
\end{thm}

\begin{proof} Let $n$ be the number of points. Consider the set of generator-line pairs
\begin{equation}\label{eq.small_inter}
\{ (\{a, b\}, R(ab)) : |R(ab)| \leq k+2 \} \cup \{ (\{a, b\}, B(ab)) : |R(ab)| > k+2 \}.
\end{equation}
Clearly it has size $\binom{n}{2}$. Each red line appearing in (\ref{eq.small_inter}) is repeated at most $\binom{k+2}{2}$ times. We claim that each blue line appearing in (\ref{eq.small_inter}) is not repeated. Indeed, when both $R(ab)$ and $R(cd)$ has more than $k+2$ points (there might be repreated points among $a$, $b$, $c$, and $d$), because their intersection is at most $k$, there must be another point $x \in R(ab)$, $x \not\in \{a, b\}$, such that $x \not\in R(cd)$. Such a point is in $B(cd)$ but not in $B(ab)$. So none of the blue lines will be repeated in (\ref{eq.small_inter}).

Each coloured line is repeated at most $\binom{k+2}{2}$ times, and the conclusion follows.
\end{proof}

When $k=1$, the above theorem asserts that $m_\cS \geq n^2 / 6 - O(n)$. In this sense the bound is tight, as can be seen from the Steiner triple systems.
And it is also tight in the Euclidean spaces and the real projective planes
as follows.

If the red triples are derived from an Euclidean space or the real projective plane.
Let $m_t$ be the number of red lines of size $t$.
By counting the triples $(u, v, R(uv))$ in two ways, we have
\begin{equation}\label{eq.pair_count}
\sum_{t=2}^n \binom{t}{2} m_t = \binom{n}{2}.
\end{equation}
The study of the maximum $m_3$ is called the orchard-planting problem --
one of the oldest problems in computational geometry (\cite{S}).
It is well known that $m_3$ can be as big as $\binom{n}{2} / 3 - O(n)$.
(Burr, Gr\"{u}nbaum and Sloane \cite{BGS},
and also the recent work by Green and Tao \cite{GT}.)
When this happens, by (\ref{eq.pair_count}) the number of red lines is
\[\sum_{t=2}^n m_t \leq \left( \binom{n}{2} - 3m_3 \right) + m_3 = m_3 + O(n). \]
The number of blue lines appearing in (\ref{eq.small_inter}) is bounded by
$ \binom{n}{2} - 3m_3 = O(n).$
And the blue lines that do not appear in (\ref{eq.small_inter}) are those generated by $\{a, b\}$ such that $|R(ab)| \leq 3$. These blue lines has sizes either $n$ or $n-1$; so we have at most $n+1$ such blue lines in total. Together we have $m_\cS \leq \binom{n}{2} / 3 + O(n)$ in these systems where the red triples are derived from Euclidean spaces or real projective planes.

\section{Tree metric}\label{sect.tree}

When a metric space is derived from a tree, there is exactly one path between any pair of vertices. The three point collinear relations only depend on the underlying tree.
Thus we can ignore the weights on the edges (by assigning length 1 to all edges).
In this section we solve the problem for trees completely.

We refer to the standard textbook \cite{BM} for graph-theoretical symbols and terms that are not defined in details here.

Recall that two vertices are
(non-adjacent) {\em twins} in a graph if they have the same neighbourhood (thus they are not adjacent).
It is clear that in a tree, all twins are leaves.
We will use the following notations throughout this section. Let $T$ be a tree with vertex set $V$ of size $n$,
we denote $A$ the set of leaves that
has at least a twin (i.e., at least another leaf as its sibling). And
\[A = A_1 \cup A_2 \cup ... \cup A_s\]
is the partition of twins into (non-trivial) twin classes, i.e., $A_i$'s are the maximal set of vertices that
are twins to each other. $s \geq 0$ is the number of twin classes.
$V \setminus A$ consists of all
the non-leaves as well as all the leaves that is the only leaf attached to its sole neighbour.
Write $|A| = a$ and $|V \setminus A| = b$.

For any integers $p \geq 2$ and $q \geq 2$, define the tree $S(p, q)$
to be the tree with a path of length $q$, and $p$ leaves attached to one end of the path.
It is easy to check that in $S(p, q)$, $s = 1$, $a = p$, $b = q$.

When a metric space is derived from a graph, the red triples, $f(\{a, b, c\}) = R$ (or equivalently $\{a, b, c\} \in \cH$) means one of the points is on a
shortest path between the other two. In a tree it is even simpler because there is a unique path (thus shortest) between any pair of vertices. We denote $P_{xy}$ the unique path between $x$ and $y$.
We have the following easy facts.

\begin{lem}\label{lem.tree_line} For any two vertices $x$ and $y$ in a tree,

(a) Let $P$ be the set of points on $P_{xy}$, then $P \subseteq R(xy)$ and $B(xy) \cap P = \{x, y\}$;

(b) If $x$ is in some twin class, $x\in A_i$, and $y \not\in A_i$, then $R(xy) \cap A_i = \{x\}$ and $A_i \subseteq B(xy)$.

(c) If none of $x$ and $y$ are from the twin class $A_i$, then $R(xy) \cap A_i$ is either $\emptyset$ or $A_i$, the same
if true for $B(xy) \cap A_i$.
\end{lem}

\begin{lem}\label{lem.tree_twin} In a tree $T$.
(a) If $x$, $y$, $z$, $w$ are $4$ distinct points such that $B(xy) = B(zw)$, then either $(x, z)$ and $(y, w)$ are two pairs of twins, or $(x, w)$ and $(y, z)$ are two pairs of twins.
(b) If $x$, $y$, and $z$ are $3$ distinct points such that $B(xy) = B(xz)$, then $y$ and $z$ are twins.
\end{lem}

\begin{proof} In a tree, if $x$ and $y$ are adjacent, $R(xy) = V$ and $B(xy) = \{x, y\}$. So, if $B(xy)$ is generated by another pair, $P_{xy}$ has at least one inner vertex. Suppose $P_{xy} = (x, v_1, v_2, ..., v_l, y)$, $l > 0$.

(a) When $B(xy)=B(zw)$, any three points among $x, y, z, w$ are not collinear. We claim that for any $i$, $v_i$ is on $P_{zw}$. $v_i \in B(xy) = B(zw)$, so $\{v_i, z, w\}$ is collinear. If $v_i$ is not in the middle, we may assume $z$ is on $P_{v_iw}$. Now the edges of the two paths $P_{v_ix}$ and $P_{v_iy}$ cannot both share edges with $P_{v_iw}$. If the edges of $P_{v_ix}$ and $P_{v_iw}$ are disjoinnt, together they form the path from $w$ to $x$ with $z$ in the middle, so $x \not\in B(zw)$, a contradiction. Similarly, if the edges of $P_{v_iy}$ and $P_{v_iw}$ are disjoint, we also get a contradiction.

Therefore, all the inner vertices of $P_{xy}$ are on $P_{zw}$. Similarly, all the inner vertices of $P_{zw}$ are on $P_{xy}$. So we may assume $z$ is adjacent to $v_1$ and $w$ is adjacent to $v_l$ (the other case has the similar proof). Now if $x$ has any other neighbour $u$, it is easy to check that $u \not\in B(xy)$ but $u \in B(zw)$. So $x$ is a leaf. Similarly, $y$, $z$, $w$ are all leaves, and they form two pair of twins.

(b) Consider the tree as rooted at $x$ and $p$ is the lowest common ancestor of $y$ and $z$. If $P_{py}$ has any inner vertex $v$, or $y$ has any other neighbour $v$, then $v \not\in B(xy)$ but $v \in B(xz)$. So $y$ is adjacent to $p$ and is a leaf. Similarly $z$ is adjacent to $p$ and is a leaf. So $y$ and $z$ are twins.
\end{proof}

\begin{lem}\label{lem.tree_blue} The number of blue lines is
\[ m(B) = \binom{b}{2} + a + bs + \binom{s}{2}.\]
\end{lem}
\begin{proof} Partition $V \setminus A$ into $B_1$ and $B_2$, such that $B_1$ is the set of vertices that are adjacent to some twins.
Suppose $|B_1|=b_1$ and $|B_2|=b_2$. Note that $b_1 = s$ because points in each twin class has a common neighbour,
and different twin classes do not share their common neighbours.

We enumerate the blue lines.

\begin{itemize}
\item $\binom{b}{2}$ lines generated by a pair of points in $V \setminus A$. By Lemma \ref{lem.tree_twin}, each of these is a distinct line.
\item $a$ lines generated by a point in $A$ and its neighbour. Each of them is a distinct blue line because they are of size $2$.
\item Lines generated by a point in $A$ and a point in $V \setminus A$ that are not neighbours. By Lemma \ref{lem.tree_twin}, for any $x \not\in A$ and any $1 \leq i \leq s$, we can pick any point $y_i \in A_i$, we get distinct lines $B(xy_i)$, $1 \leq i \leq s$. And it is easy to see that, when $x$ and $A_i$ are not adjacent, $B(xy)$ is the same line for any $y \in A_i$. So in this category we have $b_2 s+ b_1(s-1)$ distinct lines. (When $s=0$, $a=b_1=0$ and this is still correct.) Note that, by Lemma \ref{lem.tree_twin}, these lines are also different from those listed below, which are generated by two points in $A$.
\item Lines generated by two points in $A$ from the same twin class $A_i$. In this case the blue line is in the form
$ V \setminus \{ x \} $ for $x \in W_1$. Also, for each $x \in W_1$, $V \setminus \{x\}$ is indeed a blue line generated by a pair of twins adjacent to it. We have $b_1$ lines here.
\item Lines generated by two points in $A$ from different twin classes $A_i$ and $A_j$. Pick any $x \in A_i$ and $y \in A_j$, the complement, $\ov{B(x, y)}$ is the sole path between the common neighbour of $A_i$ and the common neighbour of $A_j$. So we have $\binom{s}{2}$ distinct lines here. And each of them has size at most $n-2$, different from lines in the above category.
\end{itemize}
Summing the numbers up concludes our proof.
\end{proof}

\begin{lem}\label{lem.tree_red} When the tree is not a star nor a path, the number of red lines
\[ m(R) \geq \binom{a}{2} + a + 1.\]
The equality holds if and only if the tree is $S(a, b)$, or it is a path with both ends attached with 2 or more twins.
\end{lem}

\begin{proof} If the tree is not a star nor a path, $n > 4$. Let $T'$ be the tree by deleting all the points in $A$. Let $h \geq 2$ be the number of leaves in $T'$. When $h = 2$, $T$ must be $S(a, b)$ or a path with both ends attached with 2 or more twins.

For each $i$, let $Q_i$ be the set of leaves in $T'$ that are not adjacent to $A_i$. So $|Q_i| \geq h-1$.

We count a subset of the red lines
\begin{itemize}
\item For any edge $xy$ in the tree, $R(xy) = V$.
\item Lines generated by two points in $A$, say $x \in A_i$ and $y \in A_j$. If $i=j$, the line is of size $3$, consists of one point in $V \setminus A$ and two points in $A$. If $i \neq j$, the line is of size at least $4$ and intersect $A$ at exactly one point in $A_i$ and one point in $A_j$. These $\binom{a}{2}$ lines are different, none of them is $V$.
\item Lines generated by a pair $x \in A$ and $y \not\in A$ such that $x \in A_i$ and $y \in Q_i$. By Lemma \ref{lem.tree_line} (b) and (c), $A_i$ is the only twin class that intersects $R(xy)$ partially at $\{x\}$ (all other twin class either does not intersect $R(xy)$, or is completely inside $R(xy)$); and it is easy to see that $R(xy) \cap Q_i = \{y\}$. Note that $|Q_i| \geq h-1$, we have at least $a(h-1)$ distinct lines. None of these lines is $V$ and, because each of them intersects exactly one twin class partially at one point, none of them are the same as the lines in the above case.
\end{itemize}

In total, there are at least $\binom{a}{2} + a + 1$ red lines. It is easy to see the equality is achieved by and only by the two types of trees as specified.
\end{proof}

\begin{thm}\label{thm.tree}
If $\cS$ is derived from a weighted tree on $n \geq 3$ vertices, then
\[m_\cS \geq \begin{cases}
\left\lfloor \frac{n^2}{4} + n + 1 \right\rfloor & \mbox{if } n \geq 6 \\
\binom{n}{2} + 1 & \mbox{if } n \leq 6
\end{cases}
.\]
The equality holds if and only if (1) $n \leq 6$ and $T$ is a path; or (2) $n \geq 6$ and $T=S(a^*, n-a^*)$,
where $a^* = (n-1)/2$ if $n$ is odd, and $a^* = n/2-1$ or $n/2$ when $n$ is even.
\end{thm}

\begin{proof} Let $T$ be a tree, $a$ be the number of leaves that has at least one twin and $b=n-a$, and $s$, $A$, $A_i$'s
as defined in this section.

Denote $t(n) = \left\lfloor n^2 / 4 + n + 1 \right\rfloor$.
It is easy to check that
\begin{itemize}
\item For a star with at least $4$ points, $m(R) = \binom{n-1}{2} + 1$, $m(B) = n$, and $m = m(R)+m(B) = \binom{n}{2}+2$.
\item For a path with at least $3$ vertices, $m(R) = 1$, $m(B) = \binom{n}{2}$, and $m = \binom{n}{2} + 1$.
\item For $S(a, b)$, $m(R) = \binom{a}{2}+a+1$ (Lemma \ref{lem.tree_red}), $m(B) = \binom{b}{2} + a + b$ (Lemma \ref{lem.tree_blue}), and $m = \binom{a}{2} + \binom{b}{2} + 2a + b + 1$. For $S(a^*, n-a^*)$, $m = t(n)$.
\end{itemize}
The statement is easy to check when $n \leq 6$. And when the tree is a path or a star with $n > 6$, $m > t(n)$. In the rest of the proof we assume $n > 6$ and the tree is not a path nor a star.

{\em Case 1.} $s = 0$. Lemma \ref{lem.tree_blue} reads that $m(B) = \binom{n}{2}$; so
\[ m \geq \binom{n}{2} + 1 > t(n). \]

{\em Case 2.} $s > 0$. By Lemma \ref{lem.tree_red} and \ref{lem.tree_blue},
\[ m_\cS = m(R) + m(B) \geq \binom{a}{2} + \binom{b}{2} + 2a + b + 1, \]
the equality holds only when $s=1$ and $T = S(a, b)$.

When $a+b = n$, the quantiy
\[ \binom{a}{2} + \binom{b}{2} + 2a + b + 1 = a^2 - (n-1)a + \frac{1}{2}(n^2+n+2)\]
achieves its minimum if and only if $a$ is (one of the) $a^*$ as specified in the statement.
\end{proof}

\section*{Acknowledgement}

We would like to thank  Pierre Aboulker, Adrian Bondy, Va\v sek Chv\'atal, Guangda Huzhang, and Kuan Yang for discussions on this work.

\end{document}